\documentclass[a4paper,final,texthight24cm,11pt]{article}

\usepackage{amssymb}
\usepackage{graphicx}
\usepackage{amscd}
\usepackage{amsmath,amsthm,amsxtra}
\newtheorem{thm}{Theorem}[section]

\newtheorem{lemma}[thm]{Lemma}
\newtheorem{prop}[thm]{Proposition}
\newtheorem{proposition}[thm]{Proposition}

\newtheorem{definition}[thm]{Definition}

\newtheorem{remark}[thm]{Remark}

\newtheorem{examples}[thm]{Examples}

\numberwithin{equation}{section}

\begin{document}

\noindent
{{\Large\bf About the infinite dimensional skew and obliquely reflected Ornstein-Uhlenbeck 
process\footnote{This research was supported by DFG through Grant Ro 1195/10-1 and by NRF-DFG Collaborative Research program through the National Research Foundation of Korea NRF-2012K2A5A6047864.}}

\bigskip
\noindent
\centerline{{\bf Michael R\"ockner, Gerald Trutnau}}
\\

\begin{abstract}Based on an integration by parts formula for closed and convex subsets $\Gamma$ of a separable real Hilbert space $H$ with respect to a Gaussian measure, 
we first construct and identify the infinite dimensional analogue of the obliquely reflected Ornstein-Uhlenbeck process (perturbed by a bounded drift $B$) by means of a Skorokhod type decomposition. 
The variable oblique reflection at a reflection point of the boundary $\partial \Gamma$ is uniquely described through a reflection angle and a direction in the 
tangent space (more precisely through an element of the orthogonal complement of the normal vector) at the reflection point. In case of normal reflection at the boundary of a regular convex set 
and under some monotonicity condition on $B$, we prove the existence and uniqueness of a strong solution to the corresponding SDE. 
Subsequently, we consider an increasing sequence $(\Gamma_{\alpha_k})_{k\in\mathbb{Z}}$ 
of closed and convex subsets of $H$ and the skew reflection problem at the boundaries of this sequence. We present concrete examples and obtain as a special case the infinite dimensional analogue of the 
$p$-skew reflected Ornstein-Uhlenbeck process.

\end{abstract}



\section{Introduction}
The construction and the analysis of stochastic differential equations with reflection on a closed subset $\Gamma$ of a finite 
dimensional state space is meanwhile a well-established topic and one can say that there are at least two approaches to it. 
The first approach is probabilistic, which can be realized by local time calculus (see e.g. \cite[chapter VI]{RY99}, \cite{Pe07}, \cite{HS81}), 
by solving the Skorokhod problem (\cite{LS84}, \cite{DI93}, \cite{C98}, \cite{DuRa99}), or by penalization methods (\cite{LMS81}, \cite{PW94}). 
In particular, the approach is realized through \lq\lq direct\rq\rq\ stochastic calculus. \\
The second approach is an indirect analytic one and uses as a main ingredient an integration by parts (hereafter IBP) formula on $\Gamma$, such as the weak Gauss Theorem 
for regular domains $\Gamma$. The connection to the corresponding diffusion can be  made by associating the infinitesimal generator of 
the diffusion with adequate boundary conditions through the IBP formula to a bilinear form, a so-called 
(generalized) Dirichlet form. Then, one uses the corresponding machinery (see \cite{K87}, \cite{BaHsu91}, \cite{FOT11}, \cite{MR92}, \cite{St99}, \cite{O13}, \cite{Tr03} and references therein). 
The approach is indirect because one has to establish the connection between an analytic boundary term, a surface measure on the boundary $\partial\Gamma$ 
that appears in the IBP formula, and a local time, i.e. the reflection term that appears in the stochastic differential equation (hereafter SDE) of the diffusion. This connection is established 
through Revuz's correspondence and originates from \cite{R70} (cf. \cite[Theorem 5.1.4]{FOT11} and \cite[Chapter X]{RY99}). \\
While direct stochastic calculus is often dimension dependent and is the most powerful in dimension one, the indirect approach is dimension 
independent and only relies on an integration by parts formula. In particular, the Revuz correspondence has been shown to be also valid on 
infinite dimensional state spaces (see \cite[VI. Theorem 2.4]{MR92}, and \cite[Theorem 3.1]{Tr00}). \\
In this paper, we will use the theory of Dirichlet forms and the observation that a concrete IBP formula 
on an infinite dimensional space (namely (\ref{7}) below) together with the Revuz correspondence will lead to the explicit identification of weak solutions to 
infinite dimensional analogues of well-known reflected SDEs, like oblique and skew reflection (see \cite{HS81}, \cite{DI93}).\\ 
An IBP formula on a closed subset of an infinite dimensional linear space first appears in \cite{Z01} (see also \cite{Z02}). 
There it was shown that the solution of some SPDE on
$L^2((0,1), d\xi)$ with reflection, which is forced to stay positive by penalization (see \cite{NP92}), admits a corresponding IBP formula with respect to the $3d$-Bessel bridge $\nu$, and  
$\nu$ has closed and convex support on the nonnegative functions in $L^2((0,1), d\xi)$. 
In \cite{RZZ12} the notion of BV functions in a Gelfand triple is defined, which is an extension of the definition of BV functions in a Hilbert space 
in \cite{AMMP10} (see also \cite{F00}, \cite{FH01} in case of the abstract Wiener space). In  \cite{RZZ12} they consider the Dirichlet form determined by
\begin{eqnarray}\label{55}
{\cal E}^{\rho}(u,v)=\frac{1}{2}\int_{H}\langle Du, Dv\rangle \rho(z)\mu(dz), \ \  u,v\in C^1_b(\text{supp}(\rho d\mu)) 
\end{eqnarray}
where  $H$ is a Hilbert space, $D$ the Fr\'echet derivative, $\rho$ a strictly positive and integrable BV function,  
$\mu$ a Gaussian measure in $H$, and obtain a Skorokhod representation for the associated process (a reflected Ornstein-Uhlenbeck (OU) process), 
if $\rho=\mathbb{I}_\Gamma$ and $\Gamma$ is a closed and convex set. For a nice convex set with non-empty interior, 
as in (${\it0}4$) of Subsection \ref{s3.1} below, the corresponding generator of the Dirichlet form is identified 
and hence an IBP formula is derived in \cite{BDT09}, \cite{BDT10}. In this case $\mathbb{I}_\Gamma\in BV(H,H)$ (cf. Lemma 3.7(ii) and for the definition of $BV(H,H)$ see Section \ref{sec2}). In \cite{RZZ12} the generalized weak Gauss Theorem reads as
\begin{eqnarray*}
\int_{\Gamma} D^* G(z)\mu(dz)=-\int_{\partial \Gamma} \,   _{H_1}\langle G(z), \eta_{\Gamma}\rangle_{ H_1^*} \|\partial \Gamma\|(dz) 
\end{eqnarray*}
where $(H_1,H,H_1^*)$ is a Gelfand triple of Hilbert spaces, $\eta_{\Gamma}:H\to H_1^*$ the \lq\lq exterior normal\rq\rq\  to $\partial\Gamma$, 
$\|\partial \Gamma\|$ a finite positive measure, i.e. the \lq\lq surface measure\rq\rq\ on $\partial \Gamma$ (for the precise definitions see explanations around (\ref{7})). 
A concrete situation, where the generalized weak Gauss Theorem holds and where  $\mathbb{I}_\Gamma\in BV(H,H_1)\setminus BV(H,H)$ (for the definition of $BV(H,H_1)$ see Section \ref{sec2}), is analyzed in \cite[Section 6]{RZZ12} (see Example \ref{exam1}(ii) below).\\ 
The organization, contents and the main results of this paper are as follows. In Section \ref{sec2} we introduce the general setting and the framework of BV-functions in a Gelfand triple from \cite{RZZ12}. Then, we choose 
${\cal E}^{\rho}$ as a reference Dirichlet form, where we suppose that $\rho$ satisfies either (\ref{2}) or (\ref{3})  
and that there exists a local (possibly non-symmetric) Dirichlet form $\mathcal{E}$ which is equivalent to ${\cal E}^{\rho}$ in the sense of (\ref{H2}). 
In case of (\ref{2}), $\mathcal{E}$ 
will turn out to be the Dirichlet form of the (countably) skew reflected OU-process and in 
case of (\ref3}), $\mathcal{E}$ 
will turn out to be the Dirichlet form of the obliquely reflected OU-process. Assumption (\ref{H2}) guarantees that the $\mathcal{E}^\rho$-smooth measures and $\mathcal{E}$-smooth measures are the 
same (see Remark \ref{r2}(i)) and that $\mathcal{E}^\rho$ and $\mathcal{E}$ share important common properties (see Theorem \ref{th1}). Proposition \ref{p3} constitutes a necessary tool to identify 
 $M^{[l]}$ and $N^{[l]}$ in the Fukushima decomposition (\ref{6}) related to $\mathcal{E}$, i.e. Proposition \ref{p3} will serve later for the identification of Skorokhod type decompositions in case of an explicitly specified $\mathcal{E}$. A further necessary tool is the fundamental formula (\ref{9}) of Subsection \ref{s2.1}. It follows in a straightforward manner from the IBP formula (\ref{7}) and is the main analytic tool in \cite{RZZ12} to obtain weak existence of the normally reflected OU-process. The fundamental formula (\ref{9}) is used as a basis to derive the IBP formulas in case of oblique and skew reflection. These are given in Propositions  \ref{p6} and \ref{l14}  and are fundamental for our main results.  \\
In Section \ref{sec3}, we introduce the basic setting for oblique reflection on a convex and closed subset $\Gamma$ of $H$, 
such that $\mathbb{I}_\Gamma\in BV(H,H)$. We fix a fully antisymmetric dispersion matrix $\breve{A}=(\breve{a}_{ij})_{i,j\geq1}$ satisfying (${\it0}1$)-(${\it0}3$) (see beginning of Section \ref{sec3}). Here, we need the boundedness of $\Gamma$ in order to guarantee the boundedness of the antisymmetric part of the logarithmic derivative $\beta^{\mu,\breve{A}}$ of $\mu$ with respect to $\breve{A}$ (cf. Remark \ref{r7}). Subsequently, the Dirichlet form $\mathcal{E}^{\Gamma,\bar{A}}$, where $\bar{A}:= \breve{A} + \textup{Id}$ is introduced. 
$\mathcal{E}^{\Gamma,\bar{A}}:=\mathcal{E}$ is sectorial and coincides with $\mathcal{E}^{\Gamma}:=\mathcal{E}^{\mathbb{I}_\Gamma}$ on the diagonal, thus satisfies (\ref{H2}) (see Lemma \ref{l4}), and serves in a first step to obtain via the IBP formulas Lemma \ref{l5} and Proposition \ref{p6}, the obliquely reflected OU-process with drift $\beta^{\mu,\breve{A}}$ in Proposition \ref{p3.4}. We then use a Girsanov transformation in infinite dimensions to remove 
$\beta^{\mu,\breve{A}}$ in Proposition \ref{p3.4} and by this, we obtain the obliquely reflected OU-process with bounded drift $B$ in Theorem \ref{th3.6}, which is the first main result of this paper. In Subsection \ref{s3.1}, we consider a (bounded) regular convex set given by (${\it0}4$). Then it is known that $\mathbb{I}_\Gamma\in BV(H,H)$ and so Theorem \ref{th3.6} can be concretely reformulated as Theorem \ref{th3.7a} under this setting. In Remark \ref{r16} the solution of Theorem \ref{th3.6} is characterized as an obliquely reflected process. 
In Subsection \ref{s3.2}, we consider the normally reflected case, i.e. $\breve{A}\equiv 0$, so we may drop the assumption on boundedness of $\Gamma$ (see Remark \ref{r7}). Noting that we construct the weak solution in Theorem \ref{th3.7a} without the restrictive assumption of positive $\mu$-divergence on $B$ as in \cite[(5.15)]{RZZ12}, we obtain Theorem \ref{uniqueness} just assuming $B$ to be uniformly bounded and to satisfy the monotonicity condition of
Theorem \ref{mono}. This generalizes \cite[Theorem 5.19]{RZZ12}.\\
In Section \ref{sec4} under the conditions \eqref{S1}-\eqref{S3} and the corresponding definition of $\rho$ as a step function of countably many sets 
which are differences of closed and convex sets, we obtain the IBP formula of Proposition \ref{l14}. 
Then with the help of Lemma \ref{l13}, we are able to identify the countably skew reflected OU-process 
in Theorem \ref{t4.4} which is the second main result of this paper, where the occurring local times are normalized according to the interpretation of Remark \ref{r14}. 
We present concrete examples where Theorem \ref{t4.4} can be applied in Examples \ref{exam1}. Finally, in Remark 
\ref{bouzam}, we mention the related paper \cite{BouZ14}, which is to our knowledge the sole other work about skew reflection in infinite dimensions.

\section{Framework and preliminaries}\label{sec2}
Let $H$ be separable real Hilbert space with inner product $\langle \cdot, \cdot \rangle$, norm $|\cdot|$ and Borel $\sigma$-algebra $\mathcal{B}(H)$.\\
Let $A: D(A) \subset H \rightarrow H$ be a linear self-adjoint operator on $H$ such that $A\geq\delta \operatorname{Id}$, i.e. $\langle Ax, x \rangle \geq \delta |x|^2$ for any $x \in D(A)$ and some $\delta>0$.
We further suppose that $A^{-1}$ is of trace class. In particular there exists an orthonormal basis $\{e_j, j \in \mathbb{N} \}$ of $H$ consisting of eigenfunctions for $A$ with corresponding real eigenvalues $\alpha_j, j \in \mathbb{N}$,
i.e. 
\begin{equation*}
 Ae_j = \alpha_je_j, \quad \forall j \in \mathbb{N}\;.
\end{equation*}
Consequently, $\alpha_j \geq \delta$ for any $j\in \mathbb{N}$.\\
We denote by $\mu$ the Gaussian measure on $H$ with mean zero and covariance operator
\begin{equation*}
 Q:= \frac{1}{2} A^{-1}\;.
\end{equation*}
Since $A$ is strictly positive, $\mu$ is nondegenerate and has full support on $(H, \mathcal{B}(H))$. The corresponding $L^p$-spaces, $p\in[1,\infty]$, with the usual norms $\| \cdot\|_p$ are denoted by $L^p(H;\mu)$.
We denote by $D\varphi:H \rightarrow H^{*}(\equiv H)$ the Fr\'{e}chet derivative of $\varphi:H\rightarrow \mathbb{R}$ and let $C^1_b(H)$ be the set of all bounded (Fr\'{e}chet) differentiable functions that have bounded
(Fr\'{e}chet) derivatives.\\
For $K\subset H$ define 
\begin{equation*}
 C^1_b(K):= \left\{ f: K\rightarrow \mathbb{R}\;| \; \; \exists g \in C_b^1(H) \;\textup{ with } \;f=g \textup{ on } K\right\}\;.
\end{equation*}
Moreover, if $\varphi$ is Fr\'{e}chet differentiable, we write 
\begin{equation*}
 \partial_j \varphi := \langle D \varphi, e_j \rangle\;, \quad j \geq1\;.
\end{equation*}
Next, we want to introduce the bounded variation functions in a Gelfand triple from \cite{RZZ12}. For this, we let as in \cite{FH01}
\begin{equation*}
 A_{\frac{1}{2}}(x):= \int_0^x \left( \log (1+s) \right)^{\frac{1}{2}} \mathrm{d} s, \quad x \geq 0
\end{equation*}
and
\begin{equation*}
 L(\log L)^{\frac{1}{2}}(H, \mu):= \left\{ \;f: H \rightarrow \mathbb{R}\; | f \textup{ is } \mathcal{B}(H)\textup{-measurable and } A_{\frac{1}{2}} (|f|) \in L^1(H, \mu) \right\}\;.
\end{equation*}
It is well-known that $L(\log L)^{\frac{1}{2}}(H, \mu)$ is a Banach space with norm 
\begin{equation*}
 \|f\|_{L(\log L)^{\frac{1}{2}}}=\inf \left\{ \alpha >0 \; \Big| \; \int\limits_H A_{\frac{1}{2}} \left( \frac{|f|}{\alpha}\right) \mathrm{d} \mu \leq 1 \right\} \;.
\end{equation*}
Let $(c_j)_{j\in \mathbb{N}}$ be a sequence in $[1, \infty)$ and let 
\begin{equation*}
 H_1:= \left\{ x \in H \;\Big| \;\sum\limits_{j\geq1} c_j^2 \langle x, e_j \rangle^2 < \infty \right\}
\end{equation*}
with inner product
\begin{equation*}
 \langle x, y \rangle_{H_1} := \sum\limits_{j\geq1} c_j^2 \langle x, e_j \rangle \langle y, e_j \rangle \;.
\end{equation*}
Then $(H, \langle \cdot, \cdot \rangle_{H_1})$ is a Hilbert space such that $H_1 \subset H$ continuously and densely. Identifying $H$ with its dual $H^{*}$, we obtain continuous and dense embeddings
\begin{equation*}
 H_1 \subset H( \equiv H^{*} ) \subset H^{*}_1\;.
\end{equation*}
In particular ${ }_{H_1}\langle \; \cdot \;,\; \cdot \;\rangle_{H^{*}_1}$ coincides with $\langle \cdot, \cdot \rangle$ when restricted to $H_1 \times H$ and $( H_1, H, H^{*}_1)$ is a Gelfand triple.
Define a family of $H$- valued functions on $H$ by 
\begin{align*}
 (C_b^1)_{D(A) \cap H_1} := \Bigg\{\; G \; \;  \Bigg| \;G(z)= &\sum\limits_{j=1}^m g_j (z) l^j, \; z \in H, \;\\ 
 &g_j \in C^1_b(H), \;l^j \in D(A) \cap H_1 \Bigg\}
\end{align*}
and let $D^{*}$ be the adjoint of $D: C_b^1(H)\subset L^2(H;\mu) \rightarrow L^2(H,H;\mu)$ with domain 
\begin{align*}
 \operatorname{Dom} (D^{*}) := \Bigg\{  G \in L^2(H,H;\mu)\; |&\; C_b^1(H) \ni u \rightarrow \int\limits_H \langle G, Du \rangle\; \mathrm{d} \mu \\
 &\textup{ is continuous with respect to } \| \cdot \|_2 \Bigg\}\;.
\end{align*}
We have $(C^1_b)_{D(A)\cap H_1} \subset \operatorname{Dom}(D^{*})$ and for any $G \in (C^1_b)_{D(A)\cap H_1}, \; f \in C^1_b(H)$
\begin{equation}\label{1}
 \int\limits_H D^{*} G(z) f(z) \mu (\mathrm{d}z)= \int\limits_H \langle G(z), Df(z) \rangle \mu(\mathrm{d}z)
\end{equation}
For $\rho \in L(\log L)^{\frac{1}{2}}(H, \mu)$ set
\begin{align*}
  V(\rho):= \sup\limits_{G\in(C_b^1)_{D(A) \cap H_1}, |G|_{H_1} \leq 1 } \int\limits_H D^{*}G(z) \rho(z)\; \mu (\mathrm{d}z)\;.
\end{align*}
Then
\begin{equation*}
 \rho\in BV(H,H_1):=\left\{ \rho\in L(\log L)^{\frac{1}{2}}(H, \mu) \textup {   and  } V(\rho) < \infty \right\}
\end{equation*}
is called a $BV$ function in the Gelfand triple $(H_1,H,H^{*}_1)$. Since $D(A) \cap H_1 \subset  D(A)$ and $|\cdot| \leq |\cdot|_{H_1}$ it follows that
$BV(H,H) \subset BV(H,H_1)$. From now on we will assume that $\rho:H \rightarrow \mathbb{R}$ is a $\mathcal{B}(H)$-measurable function that satisfies either 
\begin{equation*}\tag{H1a}\label{2}
 \frac{1}{c_0} \leq \rho \leq c_0 \; \; \mu-\textup{a.e. for some constant } c_0 > 1 
\end{equation*}
or
\begin{equation*}\tag{H1b}\label{3}
 \rho=\mathbb{I}_{\Gamma} \; \textup{ where } \Gamma \subset H \; \textup{ is closed and convex}.
\end{equation*}
Here $\mathbb{I}_A, \; A\subset H$, denotes the indicator function of $A$. The first case will later correspond to (countably) skew reflection and the second to oblique reflection.\\
Throughout, we set 
\begin{equation*}
 F:= \operatorname{supp} (\rho \;\mathrm{d} \mu)\;.
\end{equation*}
Then from \cite[Remark 4.1]{RZZ12}, we know that 
\begin{equation*}
 \mathcal{E}^\rho(u,v):= \frac{1}{2} \int\limits_H \langle Du, Dv \rangle \rho \;\mathrm{d} \mu\;, \quad u,v \in C^1_b(F)\;,
\end{equation*}
is closable in $L^2(F; \rho \; \mathrm{d} \mu)$ and we denote the closure by $( \mathcal{E}^\rho,D( \mathcal{E}^\rho))$.\\
Next, we suppose from now on that there exists a local (possibly non-symmetric) Dirichlet form $(\mathcal{E}, D(\mathcal{E}))$ in the sense of \cite{MR92}, such that for some constant $c_0>1$
\begin{equation*}\tag{H2}\label{H2}
 \frac{1}{c_0}\mathcal{E} (u,u) \leq \mathcal{E}^\rho(u,u)\leq c_0 \mathcal{E}(u,u)\;, \quad \forall u \in C^1_b(F)\;,
\end{equation*}
and that $(\mathcal{E}, D(\mathcal{E}))$ is given as the closure of $(\mathcal{E}, C^1_b(F))$ on $L^2(F; \rho\mathrm{d} \mu)$.
Then the following holds: 
\begin{thm}\label{th1}
 \begin{itemize}
  \item[(i)] $(\mathcal{E}, D(\mathcal{E}))$ and $ ( \mathcal{E}^\rho, D(\mathcal{E}^\rho))$ are quasi-regular 
  local Dirichlet forms in the sense of \cite[IV. Definition 3.1]{MR92} and the $\mathcal{E}$-nests and $\mathcal{E}^\rho$-nests coincide.
  \item[(ii)] $(\mathcal{E}, D(\mathcal{E}))$ is recurrent, hence conservative.
  \item[(iii)] There exists a diffusion $\mathbb{M} = (\Omega, \mathcal{M}, (\mathcal{M}_t)_{t\geq0}, (X_t)_{t\geq0}, (P_z)_{z\in F})$ associated to $(\mathcal{E}, D(\mathcal{E}))$ and $\mathbb{M}$
  has infinite life time.
 \end{itemize}
\end{thm}
\begin{proof}
 \begin{itemize}
  \item[(i)] This follows from \cite{RS92} and \cite[III. Definition 2.1]{MR92} since  the Dirichlet norms of $(\mathcal{E},D(\mathcal{E}))$ and $(\mathcal{E}^\rho,D(\mathcal{E}^\rho))$ 
  are equivalent and so the $\mathcal{E}$-nests and $\mathcal{E}^\rho$-nests as well as the $\mathcal{E}$-exceptional and 
  $\mathcal{E}^\rho$-exceptional sets coincide.
  \item[(ii)] This follows immediately from the fact that $\mathcal{E}(1,1)=0$.
  \item[(iii)] This follows from \cite[IV. Theorem 1.5]{MR92} and the conservativeness of $(\mathcal{E}, D(\mathcal{E}))$.
 \end{itemize}
\end{proof}
For the following notions and facts up to \eqref{6} below, we refer to \cite[Section VI.2]{MR92}. \\
Let $\mathcal{A}_{+}$ denote the set of all positive continuous additive functionals (PCAF's for short) of $\mathbb{M}$ as in Theorem \ref{th1} (iii) and set $\mathcal{A}:= \mathcal{A}_{+}-\mathcal{A}_{+}$,  
i.e. the set of all the CAF's of finite variation. Denote by $E_x$ the expectation with respect to $P_x$. \\
By virtue of the Revuz correspondence each $\mathcal{E}$-smooth measure $\nu \;( \nu \in S \textup{ for short})$ on $(F, \mathcal{B}(F))$ is associated to a unique $(C_t)_{t\geq0} \in \mathcal{A}_{+}$ such that
\begin{equation}\label{5}
 \lim\limits_{t\searrow0} \frac{1}{t} \int\limits_F E_x\left[ \int_0^t f(X_s) \mathrm{d} C_s \right] \;\rho(x) \mu(\mathrm{d}x) = \int\limits_F f \; \mathrm{d} \nu 
\end{equation}
for any positive $\mathcal{B}(F)$-measurable function $f: F \rightarrow \mathbb{R}$ and conversely any $(C_t)_{t\geq0}\in \mathcal{A}_{+}$ defines a unique $\nu \in S$ such that \eqref{5} holds. 
In this case, we write $\nu = \nu_C$ and $C=C^\nu$.\\
Likewise every $(C_t)_{t\geq0} \in \mathcal{A}$ corresponds to a unique $\nu \in S-S$ which is given by $\nu = \nu_{C^1}-\nu_{C^2}$ in case $C_t=C_t^1-C_t^2, \; t\geq0$, and each $\nu \in S-S$ corresponds 
to a unique $(C_t)_{t\geq0} \in \mathcal{A}$ which is given by $C_t=C_t^1-C_t^2, \; t\geq0$, in case $\nu = \nu_{C^1}-\nu_{C^2}$. 
\begin{remark}\label{r2}
\begin{itemize}
 \item[(i)] Note that by Theorem \ref{th1} (i) the class of $\mathcal{E}^\rho$-smooth measures coincides with the class of $\mathcal{E}$-smooth measures S. Thus for any $\mathcal{E}^\rho$-smooth measure $\nu$ there exists a unique $(C_t)_{t\geq0} \in \mathcal{A}_{+}$
 with $\nu=\nu_C$, i.e. \eqref{5} holds.
 \item[(ii)] From \eqref{5} and the uniqueness of the Revuz correspondence we directly see that if $\nu \in S$ is associated to $(C_t)_{t\geq0} \in \mathcal{A}_{+}$ and if $g$ is $\nu$-integrable, then
 \begin{equation*}
  \left( \int_0^t g(X_s) \mathrm{d} C_s \right)_{t\geq0} \quad \textup{ is associated to } g \mathrm{d} \nu\;.
 \end{equation*}

 \end{itemize}
\end{remark}
For each $l\in H$ the function $l(z):= \langle l,z \rangle$ belongs to $D(\mathcal{E})=D(\mathcal{E}^\rho)$ since $l(\cdot) \in L^2(\Gamma, \rho\,\mathrm{d} \mu)$. Therefore the 
$AF \;\langle l, X_t-X_0 \rangle$ admits a unique decomposition into a martingale $AF \; (M_t^{[l]})_{t\geq 0}$ of finite energy (in short $M^{[l]} \in \mathcal{\mathring{M}})$ and a CAF $(N_t^{[l]})_{t\geq0}$ of zero energy
(in short $N^{[l]}\in \mathcal{N}_c)$ such that for $\mathcal{E}$-q.e. $z \in F$ 
\begin{equation}\label{6}
 \langle l, X_t-X_0 \rangle = M^{[l]}_t+N^{[l]}_t\;, \quad t\geq 0, \; P_z-\textup{a.s.}
\end{equation}
The following identication of $N^{[l]}$ can be derived as in the finite dimensional 
case (cf. the more general Theorem 2.2 in \cite{RZZ12} and \cite[Theorem 4.5]{Tr03} where 
Proposition \ref{p3}(1) is even shown in the non-sectorial case). 
\begin{prop}\label{p3}
 \begin{itemize}
  \item[(1)] Suppose that $\nu= \nu_1- \nu_2, \; \nu_i \in S$ and finite, $i=1,2$, and 
  \begin{equation*}
    - \mathcal{E}(l(\cdot), v ) = \int\limits_\Gamma v(z) \nu\; (\mathrm{d}z) \quad \forall v \in C_b^1(F)\;.
  \end{equation*}
Then \begin{itemize}
\item[(i)] $ \int\limits_F E_x [N^{[l]}_t] \rho(x)\; \mu (\mathrm{d} x) <\infty, \quad \forall t \geq0 $
\item[(ii)] $N^{[l]}_t = C_t^{\nu_1}-C_t^{\nu_2}$
 \end{itemize}
\item[(2)] The quadratic variation of $M^{[l]}$ satisfies 
\begin{equation*}
 \langle M^{[l]} \rangle_t = t |l|^2\;, \quad t\geq0\;.
\end{equation*}
\end{itemize}
\end{prop}

\subsection{The integration by parts formula for $\rho=\mathbb{I}_{\Gamma}$}\label{s2.1}
In this subsection, we fix according to \eqref{3} 
\begin{equation*}
  \rho = \mathbb{I}_\Gamma\;,
\end{equation*}
where $\Gamma \subset H$ is closed and convex and suppose that 
\begin{equation*}
 \rho \in BV(H,H_1).
\end{equation*}
For short we denote the quasi-regular Dirichlet form $(\mathcal{E}^{\mathbb{I}_\Gamma}, D(\mathcal{E}^{\mathbb{I}_\Gamma}))$ of Theorem \ref{th1} (i) by $(\mathcal{E}^\Gamma, D(\mathcal{E}^\Gamma))$. Then the following is known from
\cite[Theorem 3.1 and 4.2, Remark 4.1]{RZZ12}:

There exists a (unique) positive finite measure $\|\partial \Gamma \|$ on $(\Gamma, \mathcal{B}(\Gamma))$ where $\mathcal{B}(\Gamma)$ is the trace $\sigma$-algebra of $\mathcal{B}(H)$ on $\Gamma$ 
and a $B(H)$-measurable map $\eta_\Gamma : H \rightarrow H^{*}_1$ (unique $\| \partial \Gamma \|$-a.e. on $\partial \Gamma$) such that $| \eta_\Gamma (z) |_{H^*_1}= 1 \; \; \|\partial \Gamma \|$-a.e. $ z \in \partial \Gamma$,
$\|\partial \Gamma \|(\Gamma) = V(\mathbb{I}_{\Gamma})$ and for any $G \in (C_b^1)_{D(A)\cap H_1}$
\begin{equation}\label{7}
  \int\limits_\Gamma D^* G(z) \;\mu (\mathrm{d} z) = - \int\limits_{\partial \Gamma } \;{ }_{H_1}\langle \; G(z) \;,\; \eta_\Gamma(z) \;\rangle_{H^{*}_1} \;\| \partial \Gamma \|\;(\mathrm{d}z).
\end{equation}
Furthermore, $\|\partial \Gamma \|$ is $\mathcal{E}^\Gamma$-smooth and supported by $\partial \Gamma$.\\
The process associated with $(\mathcal{E}^{\Gamma}, D(\mathcal{E}^\Gamma))$ which is the normally reflected infinite dimensional Ornstein-Uhlenbeck (OU for short) process has been studied in \cite{RZZ12}.
Here on the basis of \eqref{7}, we intend to construct an obliquely reflected and (countably) skew reflected OU-process. In order to do so, we will extensively use the integration by parts formula \eqref{9}
below that involves a generalized Gauss formula:\\
Indeed for $g \in C^1_b(H)$ and $l \in D(A) \cap H_1$, we get from \eqref{1} and the well-known formula for the logarithmic derivative of $\mu$ that
\begin{equation}\label{8}
 D^*(g\cdot l)(z) = - \langle l, Dg(z) \rangle + 2 g(z) \langle Al, z \rangle\;, \quad z \in H
\end{equation}
and inserting \eqref{8} into \eqref{7} gives
\begin{align}\label{9}
 -\mathcal{E}^\Gamma ( l(\cdot), g ) &= - \frac{1}{2} \int\limits_\Gamma \langle l, D g(z) \rangle \; \mu( \mathrm{d} z)\\
 &= - \int\limits_\Gamma g(z) \langle Al, z \rangle \; \mu( \mathrm{d}z) -\int\limits_{\partial \Gamma} g(z)\; { }_{H_1}\langle \; l \;,\; \eta_\Gamma(z) \;\rangle_{H^{*}_1}\;\frac{\|\partial \Gamma \|}{2}(\mathrm{d} z) \notag
\end{align}

\section{The obliquely reflected OU-process}\label{sec3}
In this section, we fix $\rho=\mathbb{I}_\Gamma$ as in Subsection \ref{s2.1} and assume additionally 
 $$\Gamma \text{ is bounded and }\;\mathbb{I}_\Gamma \in BV(H,H)\,.$$
Let $\mathcal{L}_\infty (H):= L(H,H)$ denote the space of bounded linear operators on $H$ and let $\|\cdot \|$ denote the operator norm. Let 
$\breve{A}: \Gamma \rightarrow \mathcal{L}_\infty(H)$ be a map such that
\begin{equation*}\tag{{\it 0}1}
  \sup\limits_{x \in \Gamma}\|\breve{A} (x) \|_{\textup{Fro}} < \infty\;,
\end{equation*}
where \begin{equation*}
       \| \breve{A}(z)\|^2_{\textup{Fro}}:= \sum\limits_{i,j \geq1} | \breve{a}_{ij}(z)|^2, \quad z \in \Gamma\,,
      \end{equation*}
and $\breve{a}_{ij}: \Gamma \rightarrow \mathbb{R}\,, i,j\geq1$, are the uniquely determined functions such that 
\begin{equation*}
 \breve{A}(\cdot) e_j = \sum\limits_{i\geq 1} \breve{a}_{ij}(\cdot) e_i \quad \textup{ for any } j\geq 1\;.
\end{equation*}
Note that $\| \breve{A}(\cdot)\| \leq \|\breve{A}(\cdot)\|_{\textup{Fro}}$ everywhere on $H$, thus 
\begin{equation*}
 \| \breve{A}(\cdot)\|_{L^\infty (\Gamma; \mu)} \leq \sup\limits_{x\in \Gamma} \|\breve{A}(x)\|_{\textup{Fro}}\,.
\end{equation*}
We suppose that $\breve{A}=(\breve{a}_{ij})_{i,j\geq1}$ is fully antisymmetric, that is 
\begin{equation*}\tag{{\it 0}2}
 \breve{a}_{ij} = - \breve{a}_{ji} \quad \textup{ for any } i,j\geq1\;.
\end{equation*}
We further assume
\begin{itemize}
 \item[(i)] $\breve{a}_{ij} \in C^1_b(\Gamma)$ for any $i,j\geq1\;.$ $\hspace{37.5ex} ({\it0}3)$
 \item[(ii)] $\beta^{\mu, \breve{A}}:= \sum\limits_{j\geq1} \beta_{e_j}^{\mu, \breve{A}}$ with 
 \begin{equation*}
  \beta_{e_j}^{\mu, \breve{A}}:= \sum\limits_{i\geq1} \left( \frac{\breve{a}_{ij}}{2} \beta^\mu_{e_i} + \frac{\partial_i \breve{a}_{ij}}{2}\right) \cdot e_j\;, \quad j\geq1,
 \end{equation*}
and
\begin{equation*}
 \beta^\mu_{e_i} := - 2\; \langle A e_i, \cdot \;\rangle\;, \quad i\geq1,
\end{equation*}
satisfies
\begin{equation*}
 \mu\textup{- ess}\sup\limits_{x\in \Gamma} \sum_{j\geq1} \left( \sum_{i\geq1} \left| \frac{\breve{a}_{ij}(x)}{2}\beta^\mu_{e_i}(x) + \frac{\partial_i\breve{a}_{ij}(x)}{2}\right|\right)^2 < \infty\,,
\end{equation*}
i.e.
\begin{equation*}
 \beta^{\mu,\breve{A}} \in L^\infty (\Gamma, H; \mu)\;.
\end{equation*}
\end{itemize}
\begin{remark}\label{r7}
 Since the components $\beta^{\mu}_{e_i},\, i \geq 1$, are linear in $x \in \Gamma$, it may be impossible to obtain $\beta^{\mu, \breve{A}} \in L^\infty(\Gamma, H; \mu)$ if $ \Gamma$ is unbounded. 
 That is the sole reason why we assume $\Gamma$ to be bounded at the beginning of Section \ref{sec3}. 
\end{remark}

Define 
\begin{equation*}
 \mathcal{E}^{\Gamma, \breve{A}}(u,v):= \frac{1}{2} \int\limits_{\Gamma} \; \langle \breve{A} (Du), Dv \rangle\; \mathrm{d} \mu\;, \quad u,v \in C^1_b(\Gamma),
\end{equation*}
and for 
\begin{equation*}
 \bar{A}:= \breve{A} + \textup{Id},
\end{equation*}
Id being the identity operator on $H$, let
\begin{equation*}
 \mathcal{E}^{\Gamma, \bar{A}}(u,v):= \frac{1}{2} \int\limits_{\Gamma} \; \langle \bar{A} (Du), Dv \rangle\; \mathrm{d} \mu, \quad u,v \in C^1_b(\Gamma)\;.
\end{equation*}
Then it is easy to see that for any $u,v \in C^1_b(\Gamma)$
\begin{equation*}
 \mathcal{E}^{\Gamma, \breve{A}}(u,u) = 0, \quad \textup{hence} \quad\mathcal{E}^{\Gamma, \bar{A}}(u,u) = \mathcal{E}^\Gamma(u,u)\;,
\end{equation*}
and
\begin{equation*}
 \left| \mathcal{E}^{\Gamma,\bar{A}} (u,v) \right| \leq \left( \|\breve{A}(\cdot)\|_{L^\infty(\Gamma;\mu)} + 1 \right) \left( \mathcal{E}^{\Gamma, \bar{A}}(u,u) \right)^{\frac{1}{2}}\left( \mathcal{E}^{\Gamma, \bar{A}}(v,v)\right)^{\frac{1}{2}}
\end{equation*}
These observations lead to the following:
\begin{lemma}\label{l4}
 $( \mathcal{E}^{\Gamma, \bar{A}}, C_b^1(F))$ is closable in $L^2(F; \mu)$ and the closure $(\mathcal{E}^{\Gamma,\bar{A}}, D(\mathcal{E}^{\Gamma, \bar{A}}))$ is a quasi-regular local conservative (non-symmetric) Dirichlet form on $L^2(F; \mu)$ 
 whose norm is equal to the norm of $(\mathcal{E}^\Gamma, D(\mathcal{E}^\Gamma))$, i.e.
 \begin{equation*}
  \|u\|_{D(\mathcal{E}^\Gamma)} = \| u \|_{D(\mathcal{E}^{\Gamma, \bar{A}})}\;, \quad \forall u \in D(\mathcal{E}^\Gamma)=D(\mathcal{E}^{\Gamma,\bar{A}})
 \end{equation*}
where for any Dirichlet form $(\mathcal{E}, D(\mathcal{E}))$ on some $L^2$-space with inner product $(\cdot,\cdot)_{L^2}$, we set 
\begin{align*}
 \mathcal{E}_\alpha (u,v):=\mathcal{E}(u,v) + \alpha (u,v)_{L^2}\;, \quad u,v \in D(\mathcal{E}), \alpha>0\;,
\end{align*}
and
\begin{equation*}
 \|u\|_{D(\mathcal{E})} := \mathcal{E}_1(u,u)^{\frac{1}{2}}\;.
\end{equation*}
\end{lemma}
\begin{proof}
 The sole property of $(\mathcal{E}^{\Gamma, \bar{A}}, D(\mathcal{E}^{\Gamma, \bar{A}}))$ that is not evident from the definitions is the submarkov property.
 But this follows directly from the fact that 
 \begin{equation*}
  \mathcal{E}^{\Gamma, \bar{A}}(u\wedge\alpha, u-u\wedge\alpha) \geq 0 \quad\quad \forall u \in C^1_b(F), \alpha>0,
 \end{equation*}
by \cite[I. Theorem 4.4]{MR92} and denseness of $C^1_b(F)$.
\end{proof}
In view of Lemma \ref{l4}, $(\mathcal{E}^{\Gamma, \bar{A}}, D(\mathcal{E}^{\Gamma, \bar{A}}))$ satisfies the assumptions of $(\mathcal{E}, D(\mathcal{E}))$ with respect to $(\mathcal{E}^\Gamma, D(\mathcal{E}^\Gamma)):=(\mathcal{E}^\rho,D(\mathcal{E}^\rho))$
in Section 2. We may therefore set
\begin{equation*}
                      (\mathcal{E}, D(\mathcal{E})):= (\mathcal{E}^{\Gamma, \bar{A}}, D(\mathcal{E}^{\Gamma, \bar{A}}))\;.
\end{equation*}
\begin{lemma}\label{l5}
 Let $l \in D(A) $ and $g \in C^1_b(H)$. Then 
 \begin{equation}
  -\mathcal{E}^{\Gamma, \breve{A}} (l(\cdot),g) = \int\limits_{\partial\Gamma }  \langle l, \breve{A}\eta_\Gamma \rangle g\; \frac{\|\partial \Gamma\|}{2}+ \int\limits_{\Gamma} \; \langle l, \beta^{\mu, \breve{A}} \rangle g \;\mathrm{d} \mu\;,
 \end{equation}
 where $\breve{A} \eta_\Gamma= \sum\limits_{i,j\geq1} \breve{a}_{ij}\langle  e_j, \eta_\Gamma \rangle \; e_i$.
\end{lemma}
\begin{proof}
 We have
 \begin{equation*}
  \mathcal{E}^{\Gamma, \breve{A}} ( l(\cdot),g) = \frac{1}{2} \int\limits_{\Gamma} \langle \breve{A}(z)l, Dg(z) \rangle \; \mu(\mathrm{d} z)
 \end{equation*}
and 
\begin{equation*}
 \breve{A}(z) l = \sum\limits_{i,j\geq1} \breve{a}_{ij}(z) \langle l,e_j \rangle e_i\;.
\end{equation*}
Thus by Lebesgue
\begin{equation*}
 \mathcal{E}^{\Gamma,\breve{A}}(l(\cdot),g) = \sum\limits_{i,j\geq1} \frac{1}{2} \int\limits_{\Gamma} \breve{a}_{ij}(z) \langle l,e_j \rangle\; \partial_i g(z) \mu(\mathrm{d}z)\;.
\end{equation*}
Noting that
\begin{equation*}
 \breve{a}_{ij} \langle l, e_j \rangle \partial_i g = \partial_i ( \breve{a}_{ij} \langle l,e_j \rangle g ) - \langle l,e_j\rangle g \partial_i \breve{a}_{ij}
\end{equation*}
and that $\breve{a}_{ij} \langle l,e_j \rangle g \in C^1_b(F)$ we get from \eqref{9} 
\begin{align*}
& \mathcal{E}^{\Gamma, \breve{A}}(l(\cdot),g) 
 = \sum\limits_{i,j\geq1}\left( \frac{1}{2} \int\limits_{\Gamma}  \partial_i( \breve{a}_{ij} \langle l,e_j \rangle g ) \; \mathrm{d}\mu - \frac{1}{2} \int\limits_{\Gamma} \langle l,e_j \rangle g \partial_i \breve{a}_{ij}\; \mathrm{d}\mu \right)\\
&= \sum\limits_{i,j \geq1} \left(\int\limits_{\Gamma} \breve{a}_{ij} g\langle l,e_j \rangle \,\langle A e_i,z \rangle \; \mu(\mathrm{d}z) -\frac{1}{2} \int\limits_{\Gamma} \langle l, e_j \rangle g \partial_i \breve{a}_{ij} \;  \mathrm{d}\mu\right)\\
&\; \;\;\;\;\;\quad\;\;+ \sum\limits_{i,j\geq1}\; \int\limits_{ \partial \Gamma } \breve{a}_{ij} g \langle l, e_j \rangle \;\langle \; e_i \;,\; \eta_\Gamma(z) \;\rangle \frac{\| \partial \Gamma \|}{2}( \mathrm{d}z)\\
&= -\int\limits_{\partial \Gamma } \langle l, \breve{A} \eta_\Gamma \rangle g \frac{\|\partial \Gamma\|}{2}(\mathrm{d}z)\\
&\;\; \;\;+ \int\limits_{\Gamma} \left\langle \sum_{j\geq1}\left( \sum_{i\geq1}\left( \breve{a}_{ij} \frac{\langle Ae_i,z \rangle \cdot 2}{2}- \frac{\partial_i\breve{a}_{ij}}{2} \right) \cdot e_j\right),l \right\rangle g \; \mu(\mathrm{d}z)\;.
\end{align*}

\end{proof}
Combining Lemma \ref{l5} and \eqref{9}, we get:
\begin{prop}\label{p6}
 Let $l \in D(A) $ and $g \in C_b^1(H)$. Then 
 \begin{align*}
  -\mathcal{E}^{\Gamma, \bar{A}}(l(\cdot),g)&= \int\limits_{\Gamma} g \langle l, \beta^{\mu, \breve{A}} \rangle \; \mathrm{d}\mu  - \int\limits_\Gamma g(z) \langle Al,z \rangle \mu (\mathrm{d} z)\\
 &+ \int\limits_{\partial \Gamma } g\; \langle  l, \bar{A}^{*}\nu_\Gamma \rangle\; \frac{\|\partial \Gamma \|}{2}\;, 
 \end{align*}
where $\bar{A}^{*}:=\breve{A}^{*} + \operatorname{Id}$, $\nu_\Gamma:= - \eta_\Gamma$ is the inward normal and 
$\bar{A}^{*}, \breve{A}^{*}$ denotes the transposed matrix of $\bar{A}, \breve{A}$.
\end{prop}
Proposition \ref{p6} now leads to the following intermediate result:
\begin{prop}\label{p3.4}
 There exists an $\mathcal{E}^\Gamma$-exceptional set $N\subset F$ such that for all $z \in F\setminus N$, under $P_z$ there exists an $\mathcal{M}_t$-cylindrical Wiener process $W^z$, 
such that the sample paths of $\mathbb{M}$ from Theorem \ref{th1}(iii)  on $F$
 satisfy the following:\\
 for $l \in D(A) $
 \begin{align}\label{11}
  \langle l,X_t-X_0 \rangle &= \int_0^t \langle l, \mathrm{d} W^z_s \rangle + \int_0^t \langle l, \beta^{\mu, \breve{A}}(X_s) \rangle \mathrm{d} s\\
  &-\int_0^t \langle Al, X_s \rangle \mathrm{d} s + \frac{1}{2}\int_0^t \langle  l, \bar{A}^{*}\nu_\Gamma(X_s) \rangle\; \mathrm{d}L^{\partial \Gamma}_s, \; t\geq0 \; P_z-\textup{a.s.}\notag
 \end{align}
Here $(L^{\partial \Gamma}_t)_{t\geq0} \in \mathcal{A}_{+}$ is uniquely associated to $\|\partial \Gamma\|$ via the Revuz correspondence and for all $z\in F\setminus N$
\begin{equation}\label{12}
 \int_0^t \mathbb{I}_{\partial \Gamma } \;(X_s) \mathrm{d} L^{\partial \Gamma}_s=L^{\partial \Gamma}_t, \quad t\geq0 \;\;P_z-\textup{a.s.}
\end{equation}
and $\bar{A}^{*} \nu_\Gamma(x):=\bar{A}^{*} (x) (\nu_\Gamma(x))\,,\; x \in \Gamma.$
\end{prop}
\begin{proof}
 By \cite[Theorem 4.3]{RZZ12}, we know that $\|\partial \Gamma \|$ is $\mathcal{E}^\Gamma$-smooth, hence $\mathcal{E}^{\Gamma,\breve{A}}$-smooth according to Remark \ref{r2} (i) and so by
 the Revuz correspondence there exists a unique $(L^{\partial \Gamma}_t)_{t\geq0} \in \mathcal{A}_{+}$ associated with $\| \partial \Gamma \|$. Surely $L^{\partial \Gamma}$ has property \eqref{12}
 by Remark \ref{r2} (ii) and since $\operatorname{supp} \left( \| \partial \Gamma \|\right) \subset \partial \Gamma$. \\
 Let $(e_j)_{j\geq1}$, be the orthonormal basis introduced at the beginning of Section 2. Define for $k \geq1$
 \begin{align}\label{13}
  W^z_k(t)&:= \langle e_k, X_t -z \rangle - \int_0^t \langle e_k, \beta^{\mu, \breve{A}} (X_s) \rangle \mathrm{d} s \\
  &+ \int_0^t \langle A e_k, X_s \rangle \mathrm{d} s - \frac{1}{2}\int_0^t \langle  e_k, \bar{A}^{*}\nu_\Gamma(X_s) \rangle\; \mathrm{d}L^{\partial \Gamma}_s \notag
 \end{align}
By Propositions \ref{p6} and \ref{p3} and Remark \ref{r2}, we obtain 
\begin{align}\label{14}
 N^{[e_k]}_t &= \int_0^t \langle e_k, \beta^{\mu, \breve{A}} (X_s)\rangle \mathrm{d} s - \int_0^t \langle Ae_k, X_s \rangle \mathrm{d} s \\
 &+ \frac{1}{2}\int_0^t \langle  e_k, \bar{A}^{*}\nu_\Gamma(X_s) \rangle \; \mathrm{d}L^{\partial \Gamma}_s\notag
\end{align}
Consequently, using \eqref{13}, \eqref{14} and the uniqueness of decomposition \eqref{6}, we see that
\begin{equation*}
  W^z_k(t) = M_t^{[e_k]}\;, \quad t \geq 0 \; \; P_z-\textup{a.s. for }\; \mathcal{E}^\Gamma\textup{-q.e. } z \in F,
\end{equation*}
where the $\mathcal{E}^\Gamma$-exceptional set, say $N$, can be chosen not to depend on $k\geq1$ by standard arguments. 
By direct calculations, we obtain 
$$\nu_{\langle M^{[e_i]}, M^{[e_j]}\rangle}=\delta_{ij}\mathbb{I}_\Gamma \mathrm{d}\mu = \nu_{\delta_{ij}\cdot t}\,.$$
Then using the uniqueness of the Revuz correspondence, we see that 
\begin{equation*}
 \langle M^{[e_i]},M^{[e_j]} \rangle_t = \delta_{ij} \cdot t\;. 
\end{equation*}
Thus for $z \in F\setminus N$, $W^z_k$ is an $\mathcal{M}_t$-Wiener process under $P_z$ and so \eqref{11} holds with $W^z$ being an $\mathcal{M}_t$-cylindrical Wiener process given by 
\begin{equation*}
 W^z(t) = ( W^z_k(t) e_k)_{k\geq1}\;.
\end{equation*}
\end{proof}
Next, we want to replace $\beta^{\mu, \breve{A}}$ in \eqref{11} by an arbitrary 
\begin{equation}\label{3.6}
 B \in L^\infty (\Gamma, H; \mu)\;.
\end{equation}
For this fix $T>0$. Then by general Dirichlet form theory 
\begin{equation}\label{3.7}
 \int_0^T \; |B(X_s)|^2 \mathrm{d} s \leq T \cdot \left\| \; |B(\cdot)|\; \right\|^2_{L^\infty(\Gamma; \mu)} <\infty\quad P_z-\textup{a.s.}
\end{equation}
for $\mathcal{E}^\Gamma$-q.e. $z \in F$ and we may assume that \eqref{3.7} holds for all $z \in F\setminus N$, $N$ as in Proposition \ref{p3.4}.\\
Consider the $\mathcal{M}_t$-cylindrical Wiener process $W^z$, $z \in F\setminus N$, from Proposition \ref{p3.4} and define for $z \in F \setminus N$
\begin{align}\label{ez}
 Z_t^z&:= \int_0^t \left\langle ( B - \beta^{\mu, \breve{A}}) (X_s), \mathrm{d} W^z_s \right\rangle \\
 &:= \sum\limits_{k\geq1} \int_0^t \langle e_k, (B - \beta^{\mu, \breve{A}}) (X_s) \rangle \mathrm{d} W^z_k(s), \quad t \in [0,T]. \notag
\end{align}
Then $Z^z$ is well-defined, more precisely the righthand side of \eqref{ez}
 converges $P_z$-a.s. uniformly on $[0,T]$ for all $z \in F\setminus N$. Indeed by Doob's inequality
\begin{align*}
 &E_z\left[ \sup\limits_{t\in[0,T]} \left| \sum\limits_{k=m}^n \int_0^t \langle e_k, (B - \beta^{\mu,\breve{A}})(X_s) \rangle \mathrm{d} W^z_k(s) \right|^2 \right]\\
 &\leq 4 E_z \left[ \left| \sum\limits_{k=m}^n \int_0^T  \langle e_k, (B - \beta^{\mu,\breve{A}})(X_s) \rangle \mathrm{d} W^z_k(s)\right|^2 \right]\\
 &= 4 \sum\limits_{k,l=m}^n E_z \left[ \int_0^T  \langle e_k, (B - \beta^{\mu,\breve{A}})(X_s) \rangle \mathrm{d} W^z_k(s)\cdot \int_0^T  \langle e_l, (B - \beta^{\mu,\breve{A}})(X_s) \rangle \mathrm{d} W^z_l(s)\right]\\
& = 4 \sum\limits_{k=m}^n E_z \left[ \int_0^T \left(  \langle e_k, (B - \beta^{\mu,\breve{A}})(X_s) \rangle \right)^2 \mathrm{d} s \right]
\end{align*}
which converges to $0$ for any $ z \in F \setminus N$ as $m,n \rightarrow \infty$ by \eqref{3.7}. Similarly, we can show that for any $(\mathcal{M}_t)$-stopping time $\tau \leq T$, we have 
\begin{equation*}
 E_z \left[ (Z^z_\tau)^2\right]= E_z \left[ \int_0^\tau \left|(B - \beta^{\mu,\breve{A}})(X_s)\right|^2 \mathrm{d}s\right]\;, \quad z \in F \setminus N.
\end{equation*}
Thus 
\begin{equation}
  \langle Z^z \rangle_t = \int_0^t \left| (B - \beta^{\mu,\breve{A}})(X_s)\right|^2 \mathrm{d} s, \quad t\in [0,T]\;,
\end{equation}
$P_z$-a.s. for any $z \in \setminus N$. Now define 
\begin{equation}
 \mathrm{d}\tilde{P}_z:= e^{Z^z_T-\frac{1}{2} \langle Z^z \rangle_T}\; \mathrm{d}P_z, \quad z \in F \setminus N
\end{equation}
on $(\Omega, \mathcal{M})$. Then the following holds:
\begin{prop}\label{p3.5}
 Let $z \in F \setminus N$. Then $\tilde{P}_z$ is a probability measure. Consequently,
 \begin{equation*}
  \tilde{W}^z_k(t):= W^z_k(t) - \int_0^t  \langle e_k, (B - \beta^{\mu,\breve{A}})(X_s) \rangle \mathrm{d} s\quad t \in[0,T], \; k\geq1\;,
 \end{equation*}
are independent real-valued Brownian motions on $( \Omega, \mathcal{M}, (\mathcal{M})_{t\geq0}, \tilde{P}_z)$ starting from $0$ , i.e. 
\begin{equation*}
 \tilde{W}^z_t:= (\tilde{W}^z_k(t) e_k)_{k\geq1}\,,\quad t\in [0,T]\;,
\end{equation*}
is an $\mathcal{M}_t$-cylindrical Wiener process.
\end{prop}
\begin{proof}
 The proof is completely standard (cf. \cite[Proposition I.0.5]{LR15}) and follows from the classical Girsanov theorem and L\'{e}vy's characterization theorem (see e.g. \cite[IV. (3.6) Theorem, VIII. (1.7) Theorem]{RY99})
 once we know that $\tilde{P}_z$ is a probability measure. But this is immediate from Novikov's criterion (cf. \cite[VIII. (1.6) Corollary]{RY99}), since by \eqref{3.7} 
 \begin{equation*}
  E_z \left[ e^{\frac{1}{2}\langle Z^z \rangle_T} \right] \leq e^{\frac{T}{2} \|\; | (B- \beta^{\mu, \breve{A}})(\cdot)|\; \|^2_{L^\infty(\Gamma, \mu)}} \;< \infty\;.
 \end{equation*}
 \end{proof}
Collecting the results achieved so far in this section, we obtain the following:
\begin{thm}\label{th3.6}
 Let $T>0$, $B$ as in \eqref{3.6}, $\tilde{W}^z$ as in Proposition \ref{p3.5}, $( L^{\partial \Gamma}_t)_{t\geq0}$, $N$ and all other notions as in Proposition \ref{p3.4} and $z \in F \setminus N$. Then the following holds for the sample
 paths of $\mathbb{M}$ from Theorem \ref{th1}(iii) on F:\\
 for $l \in D(A) $
 \begin{align}
  \langle l, X_t&-X_0 \rangle = \int_0^t \langle l, \mathrm{d} \tilde{W}^z_s \rangle + \int_0^t \langle l, B(X_s) \rangle \mathrm{d} s \\
  &- \int_0^t \langle Al, X_s \rangle \mathrm{d} s + \frac{1}{2}\int_0^t \langle  l, \bar{A}^{*}\nu_\Gamma(X_s) \rangle\; \mathrm{d} L^{\partial\Gamma}_s\;, \quad t \in [0,T], \; \tilde{P}_z-\textup{a.s.}\notag
 \end{align}
and
\begin{equation}
  \int_0^t \mathbb{I}_{\partial \Gamma } (X_s) \mathrm{d} L^{\partial \Gamma}_s = L^{\partial \Gamma}_t\;, \quad t \geq 0, \; \tilde{P}_z-\textup{a.s.} 
\end{equation} 
\end{thm}
\begin{proof}
 The assertions follow directly from Propositions \ref{p3.4} and \ref{p3.5} and the equivalence of $P_z$ and $\tilde{P}_z\;, z\in F \setminus N$.
\end{proof}
\subsection{Oblique reflection on a regular convex set}\label{s3.1}
In this subsection, we assume according to \cite[Hypothesis 1.1 (ii)]{BDT09} that 
\begin{itemize}
 \item [(${\it0}4$)] There exists a convex $C^\infty$-function $g: H \rightarrow \mathbb{R}$ with $g(0)=Dg(0)=0$ such that $\langle D^2 g(x) h, h \rangle \geq \gamma |h|^2 \quad \forall h \in H$ for some 
 $\gamma >0$, i.e. $D^2g$ is strictly positive definite and that 
 \begin{equation*}
  \Gamma= \left\{ x \in H \left| g(x) \leq 1 \right\}, \quad \partial \Gamma = \left\{ x \in H \right| g(x) =1 \right\},
 \end{equation*}
$D^2g$ is bounded on $\Gamma$ and $ | Q^{\frac{1}{2}} D g |^{-1} \in \bigcap\limits_{p>1} L^p (H; \mu)$.
\end{itemize}
Then the following holds:
\begin{lemma}\label{l3.7}
  \begin{itemize}
   \item[(i)] $\Gamma$ is convex and closed and there exists $\delta>0$ with $|Dg(x)| \leq \delta, \quad \forall x \in \Gamma$.
   \item[(ii)] $ \mathbb{I}_\Gamma \in BV(H,H)$.
   \item[(iii)] $\eta_\Gamma = \frac{Dg}{|Dg|}$ is the exterior normal to $\Gamma$, i.e. for $x \in \partial \Gamma$
   \begin{equation*}
    \langle \eta_\Gamma(x), y-x \rangle \leq 0, \quad \forall y \in \Gamma,
   \end{equation*}
and $| \eta_\Gamma(x)|=1$. Moreover 
\begin{equation*}
 \| \partial \Gamma \| (\mathrm{d} x) = \frac{Dg(x)}{|Q^{\frac{1}{2}} Dg(x) |} \mu_{\partial \Gamma} (\mathrm{d} x)\;,
\end{equation*}
where $\mu_{\partial \Gamma}$ is the surface measure induced by $\mu$ (cf. \cite{BDT09,BDT10,Mal97}).
\item[(iv)] $\Gamma$ has nonempty interior.
  \end{itemize}
\end{lemma}
\begin{proof}
 \begin{itemize}
  \item [(i)] Follows from \cite[Lemma 1.2]{BDT09}.
  \item[(ii)] See \cite[Theorem 5.3]{RZZ12}.
  \item[(iii)] See \cite[Theorems 5.3 and 5.4 and Remark 5.5]{RZZ12}.
  \item[(iv)] This follows, since $D^2g$ is strictly positive definite by assumption.
 \end{itemize}
\end{proof}
By Lemma \ref{l4}, $\mathbb{I}_\Gamma$ satisfies the conditions that were postulated at the beginning of Section 3. It follows thus from Theorem \ref{th3.6}:  
\begin{thm}\label{th3.7a}
 Assume $({\it 0}1)-({\it 0}4)$, $\Gamma$ bounded, and \eqref{3.6}. Let $T>0$. Then there exist an $\mathcal{E}^\Gamma$-exceptional set $N \subset F$ such that 
 $\forall z \in F \setminus N$, under $\tilde{P}_z$ there exists an $\mathcal{M}_t$-cylindrical Wiener 
process $\tilde{W}^z$ such that the sample paths of $\mathbb{M}$ from Theorem \ref{th1}(iii) satisfy the following:\\
for $l \in D(A)$
\begin{align}\label{3.13}
 \langle l, X_t-X_0 \rangle& = \int_0^t \langle l, \mathrm{d} \tilde{W}^z_s \rangle + \int_0^t \langle l, B(X_s) \rangle \mathrm{d} s\\
 &- \int_0^t \langle Al, X_s \rangle \mathrm{d} s + \frac{1}{2} \int_0^t \langle l, \bar{A}^{*} \nu_\Gamma(X_s) \rangle \mathrm{d} L^{\partial \Gamma}_s, \quad t \in [0,T], \; \tilde{P}_z\textup{-a.s.} \notag
\end{align}
and
\begin{equation}
 \int_0^t \mathbb{I}_{\partial \Gamma}(X_s) \mathrm{d} L^{\partial \Gamma}_s=L^{\partial \Gamma}_t, \quad t \geq 0, \; \tilde{P}_z\textup{-a.s.}
\end{equation}
\end{thm}
\begin{remark}\label{r16}
 Since $\nu_\Gamma(x) \in H$, for any $x \in \partial \Gamma$ the reflection term in \eqref{3.13} may be interpreted as variable oblique reflection.
 First note that the angle between $\bar{A}^{*}\nu_\Gamma(x)$ and $\nu_\Gamma(x)$ is necessarily acute, since 
 $$\langle \bar{A}^{*} \nu_\Gamma(x), \nu_\Gamma(x) \rangle = \langle \nu_\Gamma(x), \nu_\Gamma(x) \rangle =1 >0.$$
 Thus the reflection angle $\theta(x)$ at $x \in \partial\Gamma$
 may be defined as the difference of $\frac{\pi}{2}$ and the angle between $\bar{A}^{*}\nu_\Gamma(x)$
 and $\nu_\Gamma(x)$, i.e. 
 \begin{equation*}
  \theta(x):=\arcsin \left( \frac{\left\langle \bar{A}^{*} \nu_\Gamma(x), \nu_\Gamma(x) \right\rangle }{| \bar{A}^{*}\nu_\Gamma(x) |} \right)= \arcsin \left( \frac{1}{| \bar{A}^{*}\nu_\Gamma(x) |} \right)\; \in \left(0, \frac{\pi}{2}\right],
 \end{equation*}
and a direction $F(x)$ of $\bar{A}^{*}\nu_\Gamma(x)$ may be defined as follows:\\
for any $ x \in \partial \Gamma$ there exist $( z_k(x))_{k\geq1} \subset H$ such that 
\begin{equation*}
 \left\{ \nu_\Gamma(x), z_1(x), z_2(x), \dots \right\}
\end{equation*}
forms an orthonormal basis of $H$. Then
\begin{equation*}
 F(x):= \sum\limits_{k\geq1} \langle \bar{A}^{*} \nu_\Gamma(x), z_k(x) \rangle z_k(x)\;, \quad x \in \partial \Gamma\;.
\end{equation*}
\end{remark}
\begin{lemma}
 The variable oblique reflection in \eqref{3.13} is uniquely determined through 
 \begin{equation*}
  \left( \theta(x), F(x) \right), \quad x \in \partial \Gamma\, ,
 \end{equation*}
where the reflection angle $ \theta(x)$ and the direction $F(x)$ at $ x \in \partial \Gamma$ are 
given as in Remark \ref{r16}.
\end{lemma}
\begin{proof}
 Let $x \in \partial \Gamma$. We have to show that a vector $v(x) \in H$ is uniquely determined through
 \begin{itemize}
 \item[(i)] $\langle v(x), \nu_\Gamma(x) \rangle >0$,
  \item[(ii)] the values $\langle v(x), z_k(x) \rangle\, , k \in \mathbb{Z}$, and
  \item[(iii)] the value $\arcsin \left( \frac{1}{|v(x)|}\right) = \theta(x)$.
 \end{itemize}
We have 
\begin{equation*}
  v(x) = \langle v(x), \nu_\Gamma(x) \rangle \nu_\Gamma(x) + \sum\limits_{k\geq1} \langle v(x), z_k(x) \rangle z_k(x)\,.
\end{equation*}
Thus by (ii), it is enough to determine $\langle v(x), \nu_\Gamma(x) \rangle$.\\
By Parseval's identity
\begin{equation*}
 |v(x)|^2= | \langle v(x), \nu_\Gamma(x) \rangle |^2 + \sum\limits_{k\geq1} | \langle v(x), z_k(x) \rangle |^2\,,
\end{equation*}
and by (iii), we have 
\begin{equation*}
 |v(x)| = \frac{1}{\sin(\theta(x))}\;.
\end{equation*}
Thus by (i) 
\begin{equation*}
 \langle v(x) ,\nu_\Gamma(x) \rangle = \sqrt{\left(\frac{1}{\sin(\theta(x))}\right)^2 - \sum\limits_{k\geq1} | \langle v(x), z_k(x) \rangle |^2}
\end{equation*}
which concludes the proof.
\end{proof}

\subsection{Uniqueness in the case of normal reflection}\label{s3.2}
In this subsection, we assume that $\breve{A}\equiv 0$ and that $({\it0}4)$ is satisfied. According to Remark \ref{r7}, we may hence drop the assumption that $\Gamma$ is bounded.\\
We consider the following stochastic inclusion in $H$,
\begin{equation}\label{e15}
 \begin{cases} \mathrm{d} X_t = \left( A X_t + B(X_t) + N_\Gamma(X_t) \right) \mathrm{d} t \; \ni \mathrm{d} W_t\\ X_0=x,\end{cases}
\end{equation}
where $B: \Gamma \rightarrow H$ is an everywhere uniformly bounded vector field, i.e.
\begin{equation}\label{e16}
  \sup\limits_{x \in \Gamma} | B(x)| < \infty\,,
\end{equation}
$W_t$ is a cylindrical Wiener process in $H$ on a filtered probability space $(\Omega, \mathcal{F}, (\mathcal{F}_t)_{t \geq0}, P)$ and $N_\Gamma(x)$ is the normal cone to $\Gamma$ at $x \in \Gamma$, i.e. 
\begin{equation*}
 N_\Gamma(x):=\left\{ z \in H | \langle z, y-x \rangle \leq 0 \; \forall y \in \Gamma \right\}\,.
\end{equation*}
\begin{definition}\label{d17}
 A pair of $H \times \mathbb{R}$-valued and $(\mathcal{F}_t)$-adapted processes $(X_t, L_t)$, $t \in [0,T]$, is called a solution of \eqref{e15} if the following conditions hold: 
 \begin{itemize}
  \item[(i)] $X_t \in \Gamma$, for all $t \in [0,T]$ P-a.s.,
  \item[(ii)] $L$ is an increasing process with 
  \begin{equation*}
   \mathbb{I}_{\partial \Gamma} (X_s) \mathrm{d} L_s = \mathrm{d} L_s \quad \text{P-a.s.},
  \end{equation*}
and for any $l\in D(A)$, we have 
\begin{align*}
 \langle l, X_t -x \rangle =& \int_0^t \langle l , \mathrm{d} W_s \rangle - \int_0^t \langle Al, X_s \rangle \mathrm{d}s - \int_0^t \langle l, B(X_s) \rangle \mathrm{d} s\\
 & \quad - \frac{1}{2} \int_0^t \langle l, \eta_\Gamma(X_s) \rangle \mathrm{d} L_s\,, \quad \forall t \in [0,T]\,, \; \text{P-a.s.},
\end{align*}
where $\eta_\Gamma$ is the exterior normal to $\Gamma$ (see Lemma \ref{l3.7}(iii)).
 \end{itemize}
\end{definition}
\begin{thm}\label{mono}
 Suppose that $({\it0}4)$ is satisfied and that $B$ satisfies additionally to the uniform boundedness in \eqref{e16}, the monotonicity condition 
 \begin{equation*}
  \langle B(u) - B(v), u-v \rangle \geq - \alpha |u-v|^2
 \end{equation*}
for all $u,v \in \Gamma$, for some $\alpha \in [0, \infty)$ independent of $u,v$. Then the stochastic inclusion \eqref{e15} admits at most one solution in the sense of Definition \ref{d17}.
\end{thm}
\begin{proof}
 The proof is exactly the same as in \cite[Theorem 5.18]{RZZ12}.
\end{proof}
Noting that we could construct a weak solution to \eqref{3.13} without the restrictive assumption of positive $\mu$-divergence on $B$ as in \cite[(5.15)]{RZZ12}, we obtain the following generalization of \cite[Theorem 5.19]{RZZ12}.
\begin{thm}\label{uniqueness}
 Let $T >0$. There exist a Borel set $M \subset H$ with $\mu(\Gamma \cap M) = \mu(\Gamma)$ such that for every $x \in M$, \eqref{e15} has a pathwise unique continuous strong solution in the sense that for every 
 probability space $(\Omega, \mathcal{F}, (\mathcal{F}_t),P)$ with an $\mathcal{F}_t$-Wiener process $W$, there exists a unique pair of $\mathcal{F}_t$-adapted process $(X,L)$ satisfying Definition \ref{d17}
 and $P(X_0=x)=1$. Moreover $X_t \in M$ for all $t \in[0,T]$ P-a.s.
\end{thm}
\begin{proof}
 The assertion follows exactly as in the proof of \cite[Theorem 5.19]{RZZ12}.
\end{proof}

\section{The (countably) skew reflected OU-process}\label{sec4}
In this section, we will fix $\rho$ according to \eqref{2}. \\
We consider an increasing sequence of closed and convex subsets $(\Gamma_{\alpha_k})_{k\in \mathbb{Z}}$ (i.e. $\Gamma_{\alpha_k} \subset \Gamma_{\alpha_{k+1}} \quad \forall k \in \mathbb{Z}$) of $H$ such that
\begin{align*}
 \tag{S1}\label{S1}&\mathbb{I}_{\Gamma_{\alpha_k}} \in BV(H,H_1) \quad \forall k \in \mathbb{Z} \\ \tag{S2}\label{S2}
 &\lim\limits_{k \rightarrow - \infty} \mu ( \Gamma_{\alpha_k})=0 \textup{ and } \lim\limits_{k \rightarrow \infty} \mu (\Gamma_{\alpha_k}) = \mu(H)
\end{align*}
and a sequence $(\gamma_k)_{k\in \mathbb{Z}} \subset (0,\infty)$ such that
\begin{align*}\tag{S3}\label{S3}
 &\frac{1}{c_0} \leq \gamma_k \leq c_0 \quad \forall k \in \mathbb{Z} \quad \textup{for some constant } c_0>1,\\
 & \exists \; \bar{\gamma}:= \lim\limits_{k\rightarrow \infty} \gamma_k \quad \textup{ and }\notag
\end{align*}
\begin{equation}\label{4.1}
 \sum\limits_{k\in \mathbb{Z}} | \gamma_{k+1} - \gamma_k |\; \| \partial \Gamma_{\alpha_k} \|\; (\partial \Gamma_{\alpha_k}) < \infty\;.
\end{equation}
Then set 
\begin{equation*}
 \rho:= \sum\limits_{k \in \mathbb{Z}} \gamma_{k+1} \mathbb{I}_{\Gamma_{\alpha_{k+1}}\setminus \Gamma_{\alpha_k}}\;. 
\end{equation*}
By \eqref{S2} and \eqref{S3} we have 
\begin{equation}
 \frac{1}{c_0} \leq \rho \leq c_0 \quad \mu\textup{-a.e. } \; \textup{ on } H,
\end{equation}
hence \eqref{2} is satisfied. Moreover \eqref{H2} holds by setting 
\begin{equation*}
  (\mathcal{E}, D(\mathcal{E})):= ( \mathcal{E}^\rho, D( \mathcal{E}^\rho))\;.
\end{equation*}
Let $\mathbb{M}=(\Omega, \mathcal{M}, (\mathcal{M}_t)_{t\geq0}, (X_t)_{t\geq0}, (P_z)_{z\in H})$ be the conservative
diffusion of Theorem \ref{th1}(iii) that is associated with $(\mathcal{E}, D(\mathcal{E})$.
\vspace*{5pt}
\\Now, we can show the following:
\begin{lemma}\label{l13}
 \begin{itemize}
  \item[(i)] For each $k \in \mathbb{Z}, \; \| \partial \Gamma_{\alpha_k} \|$ is $\mathcal{E}$-smooth. 
  There exists hence a unique $( L^{\partial \Gamma_{\alpha_k}}_t)_{t\geq 0}\in \mathcal{A}_{+}$ associated to $\| \partial \Gamma_{\alpha_k}\|$
  via the Revuz correspondence.
  \item[(ii)] $\sum\limits_{k \in \mathbb{Z}} \frac{\gamma_{k+1} - \gamma_k}{2} \;\left\| \partial \Gamma_{\alpha_k} \right\|$ is the difference of finite $\mathcal{E}$-smooth measures and uniquely associated
  via the Revuz correspondence to $\left( \sum\limits_{k \in \mathbb{Z}} \frac{( \gamma_{k+1} - \gamma_k)}{2}\; L^{\partial \Gamma_{\alpha_k}}_t \right)_{t\geq 0}$ which converges locally uniformly in $t \geq0$.
 \end{itemize}
\end{lemma}
\begin{proof}
\begin{itemize}
 \item[(i)] Let
 $k \in \mathbb{Z}$. By \cite[Theorem 3.1(ii) and Remark 4.1]{RZZ12} $\| \partial \Gamma_{\alpha_k}\|$ is smooth with respect to the closure on $L^2( H, (\mathbb{I}_{\Gamma_{\alpha_k}}+\mathbb{I}_H) \mathrm{d} \mu)$ of
 \begin{equation*}
  \frac{1}{2} \int\limits_H \langle Du, Du \rangle (\mathbb{I}_{\Gamma_{\alpha_k}} + \mathbb{I}_H) \mathrm{d} \mu\;, \quad u,v \in C^1_b(H)\;.
 \end{equation*}
But then $\| \partial \Gamma_{\alpha_k} \|$ is also smooth with respect to $\mathcal{E}$ (see Theorem \ref{th1}(i) and its proof).
\item[(ii)] By \eqref{4.1}  and (i), $ \sum\limits_{k\in \mathbb{Z}} \frac{ \gamma_{k+1} - \gamma_k}{2} \| \partial \Gamma_{\alpha_k} \|$ is the difference of finite $\mathcal{E}$-smooth
measures and hence associated via the Revuz correspondence to a unique $(C_t)_{t\geq0}= (C_t^1-C_t^2)_{t\geq0}$, with $(C_t^1)_{t\geq0}, (C_t^2)_{t\geq0} \in \mathcal{A}_{+}$. 
Considering $C^1$ and $C^2$ separately, we may assume that $(C_t)_{t\geq0} \in \mathcal{A}_{+}$. Separating the supports of $(\| \partial \Gamma_{\alpha_k}\|)_{k\in \mathbb{Z}}$ we may even suppose that $\gamma_{k+1} - \gamma_k \geq0\quad \forall k \in \mathbb{Z}$.
Then since 
\begin{align*}
 \nu_{\sum\limits_{k=-n}^n \frac{\gamma_{k+1}-\gamma_k}{2} L^{\partial \Gamma_{\alpha_k}}}&= \sum\limits_{k=-n}^n \frac{\gamma_{k+1} - \gamma_k}{2} \;\| \partial \Gamma_{\alpha_k} \|\\
 &\leq \sum\limits_{k \in \mathbb{Z}} \frac{\gamma_{k+1}-\gamma_k}{2} \; \| \partial \Gamma_{\alpha_k} \| = \nu_C\;,
\end{align*}
it follows from general Dirichlet form theory, that for any $n\geq0$ and $t\geq0$ 
\begin{equation*}
 P_z \left( \sum_{k=-n}^n \frac{\gamma_{k+1}-\gamma_k}{2} \; L^{\partial \Gamma_{\alpha_k}}_t \leq C_t \right) = 1 \quad \textup{for }\mathcal{E}\textup{-q.e. } x \in H.
\end{equation*}
Hence by the Weierstrass $M$-test $\left( \sum\limits_{k \in \mathbb{Z}} \frac{\gamma_{k+1} - \gamma_k}{2} \; L^{\partial \Gamma_{\alpha_k}}_t \right)_{t\geq0}$ converges
$P_z$-a.s. locally uniformly in $t\geq0$ for $\mathcal{E}$-q.e. $z \in H$.\\ Thus $\left( \sum\limits_{k\in \mathbb{Z}} \frac{\gamma_{k+1}- \gamma_k}{2} \; L^{\partial \Gamma_{\alpha_k}}_t \right)_{t\geq0}$
is a PCAF of $\mathbb{M}$ and it is easy to see that its Revuz measure coincides with $\sum_{k\in \mathbb{Z}} \frac{\gamma_{k+1}-\gamma_k}{2}\; \| \partial \Gamma_{\alpha_k}\|$.
\end{itemize}
\end{proof}
For the identification of the SDE corresponding to $\mathbb{M}$, we now need to specify an integration by parts formula for $\mathcal{E}=\mathcal{E}^\rho$.
\begin{proposition}\label{l14}
 Let $l \in D(A) \cap H_1$ and $g \in C_b^1(H)$. Then 
 \begin{equation*}
  - \mathcal{E}(l(\cdot),g) = - \int\limits_H g \langle Al, z \rangle \rho \mu (\mathrm{d} z) + \sum\limits_{k \in \mathbb{Z}} \int\limits_{\partial \Gamma} g\; { }_{H_1} \langle l, \eta_{\Gamma_{\alpha_k}} \rangle_{H_1^*} ( \gamma_{k+1} - \gamma_k ) \frac{\| \partial \Gamma_{\alpha_k} \|}{2}.
 \end{equation*}
\end{proposition}
\begin{proof}
 We may rewritte $-\rho$ as 
 \begin{equation*}
   \lim\limits_{n \rightarrow \infty} \left( \gamma_{-n+1} \mathbb{I}_{\Gamma_{\alpha-n}} + \sum\limits_{k=-n+1}^{n} (\gamma_{k+1}-\gamma_k) \mathbb{I}_{\Gamma_{\alpha_k}} - \gamma_{n+1} \mathbb{I}_{\Gamma_{\alpha_{n+1}}} \right)
 \end{equation*}
and then noting that 
$$\lim\limits_{n\rightarrow \infty} \frac{1}{2} \int \langle l,Dg \rangle \gamma_{-n+1} \mathbb{I}_{\Gamma_{\alpha-n}} \mathrm{d} \mu = 0= \lim\limits_{n\rightarrow \infty} \int g \langle Al, z \rangle \gamma_{-n+1} \mathbb{I}_{\Gamma_{\alpha-n}}\mu(\mathrm{d}z)$$ and 
\begin{equation*}
 \lim\limits_{n\rightarrow \infty} \frac{1}{2} \int \langle l, Dg \rangle \gamma_{n+1} \mathbb{I}_{\Gamma_{\alpha_{n+1}}} \mathrm{d} \mu = \frac{1}{2} \int\limits_H \langle l, Dg \rangle \bar{\gamma} \mathrm{d} \mu = \int\limits_{H} g \langle Al, z \rangle \bar{\gamma} \mu(\mathrm{d}z) \;,
\end{equation*}
by \eqref{S2} and \eqref{S3}, using \eqref{4.1} we get by \eqref{9} 
\begin{align*}
 &- \mathcal{E} (l(\cdot), g ) = \lim\limits_{n\rightarrow \infty} \left( \sum\limits_{k=-n+1}^{n} \frac{1}{2} \int \langle l, Dg \rangle ( \gamma_{k+1}-\gamma_k) \mathbb{I}_{\Gamma_{\alpha_k}} \mathrm{d} \mu \right) - \frac{1}{2} \int\limits_H \langle l, Dg \rangle \bar{\gamma} \mathrm{d} \mu\\
 &= \lim\limits_{n\rightarrow \infty} \Bigg( \sum\limits_{k=-n+1}^{n}  \int\limits_{\partial \Gamma_{\alpha_k}} g\; { }_{H_1} \langle l, \eta_{\Gamma_{\alpha_k}} \rangle_{H^{*}_1} (\gamma_{k+1} - \gamma_k) \frac{\| \partial \Gamma_{\alpha_k}\|}{2}\\
&\quad \quad \quad + \int g \langle Al, z \rangle \gamma_{-n+1}\mathbb{I}_{\Gamma_{\alpha-n}} \mu (\mathrm{d}z)\\ 
&\quad \quad \quad + \sum\limits_{k=-n+1}^{n}\int g \langle Al, z \rangle (\gamma_{k+1} - \gamma_k) \mathbb{I}_{\Gamma_{\alpha_k}} \mu(\mathrm{d}z) - \int g \langle Al, z \rangle \gamma_{n+1} \mathbb{I}_{\Gamma_{\alpha_{n+1}}} \mu(\mathrm{d}z)  \Bigg)\\
& =  \sum\limits_{k\in \mathbb{Z}} \;\int\limits_{\partial \Gamma_{\alpha_k}} g \;{ }_{H_1} \langle l, \eta_{\Gamma_{\alpha_k}} \rangle_{H^{*}_1} \;( \gamma_{k+1} - \gamma_k ) \frac{\| \partial \Gamma_{\alpha_k} \|}{2} - \int\limits_H g \langle Al, z \rangle \rho  \mu(\mathrm{d}z)\;.
\end{align*}
\end{proof}
\begin{remark}\label{r14}
 An increasing process satisfying \eqref{12} can only be unique up to a constant. This constant has usually to be fixed in dimension one in order to describe local times uniquely.
 In Dirichlet form theory of energy forms the drift is given as the logarithmic derivative 
 \begin{equation*}
  \frac{D \rho}{2 \rho}
 \end{equation*}
and in dimension one in order to fix symmetric local times (cf. \cite{RY99}) one has to choose the symmetric version of $\rho$ at the boundary in the denominator 
(see for instance \cite[(19)]{ORT15} and the following explanations). Therefore, 
\begin{equation}\label{4.3}
 \frac{\gamma_{k+1}-\gamma_k}{2 \cdot \frac{1}{2}( \gamma_{k+1}+\gamma_k)} \| \partial \Gamma_{\alpha_k} \| = \frac{\gamma_{k+1} - \gamma_k}{\gamma_{k+1} + \gamma_k} \| \partial \Gamma_{\alpha_k}\|\;, \quad k \in \mathbb{Z},
\end{equation}
which represents the ``symmetric version'' of the non-absolutely continuous part of the logarithmic derivative in the sense that $\frac{1}{2}(\gamma_{k+1}+ \gamma_k)$ represents the ``symmetric value'' of $\rho$ at the boundary 
$\partial \Gamma_{\alpha_k}$, should represent the symmetric local time. We therefore set
\begin{equation}
L^{\partial \Gamma_{\alpha_k}}_t(X):= \frac{\gamma_{k+1} + \gamma_k}{2} L^{\partial \Gamma_{\alpha_k}}_t\;, \quad t \geq0, \; k \in \mathbb{Z},
\end{equation}
where $(L^{\partial \Gamma_{\alpha_k}}_t)_{t\geq0} \in \mathcal{A}_{+}$ is uniquely associated via the Revuz correspondence to $\| \partial \Gamma_{\alpha_k} \|$. It then follows that the measures in \eqref{4.3} are uniquely associated to
\begin{equation*}
 \left(\frac{\gamma_{k+1}- \gamma_k}{\gamma_{k+1} + \gamma_k} L^{\partial \Gamma_{\alpha_k}}_t(X) \right)_{t\geq0}\;, \quad k \in \mathbb{Z}.
\end{equation*}
Finally, setting
\begin{equation}\label{e4.5}
 p_k:= \frac{ \gamma_{k+1}}{\gamma_{k+1}+\gamma_k}\;, \quad k \in \mathbb{Z},
\end{equation}
we get 
\begin{equation*}
 p_k - (1 - p_k)= 2 p_k -1 = \frac{ \gamma_{k+1}-\gamma_k}{\gamma_{k+1}+\gamma_k}\;, \quad k \in \mathbb{Z}.
\end{equation*}
Note that in finite dimensions (see for instance \cite{Wei84, Ra11}) $p_k$ and $ 1- p_k$ represent the probabilities of permeability through $\partial \Gamma_{\alpha_k}$ depending on from which side of $\partial \Gamma_{\alpha_k}$
the process approaches $\partial \Gamma_{\alpha_k}$. Here, we do not intend to develop rigourously such an interpretation. However, since $(\mathcal{E}, D(\mathcal{E}))$ is recurrent and irreducible, we know indeed that the recurrent
conservative diffusion $\mathbb{M}$ passes infinitely often through each membrane $\partial \Gamma_{\alpha_k}$ as long as the interior of $\Gamma_{\alpha_k}$ is non-empty.
This can be shown similarly to the finite dimensional case (cf. \cite[Section 6]{Tr05}).
\end{remark}
Now, using the standard identification method that we used in Proposition \ref{p3.4}, we obtain from Lemma \ref{l13}, 
Proposition \ref{l14} and Remark \ref{r14}.
\begin{thm}\label{t4.4}
 There exists an $\mathcal{E}$-exceptional set $N \subset F$, such that for all $z \in H \setminus N$, under $P_z$ there exists an $\mathcal{M}_t$-cylindrical Wiener process $W^z$, such that the 
sample paths of $\mathbb{M}$ from Theorem \ref{th1}(iii)
 on $H$ satisfy:\\
 for all $l \in D(A) \cap H_1$ 
 \begin{align}\label{e4.6}
  \langle l, X_t - X_0 \rangle &= \int_0^t \langle l, \mathrm{d} W^z_s \rangle - \int_0^t \langle Al, X_s \rangle \mathrm{d} s \\
  &+\sum\limits_{k\in \mathbb{Z}}(2 p_k -1 ) \int_0^t { }_{H_1} \langle l, \eta_{\Gamma_{\alpha_k}} (X_s) \rangle_{H^{*}_1} \mathrm{d} L^{\partial \Gamma_{\alpha_k}}_s (X), \quad t \geq 0, \; P_z\textup{-a.s.} \notag
 \end{align}
where $(p_k)_{k \in \mathbb{Z}} \subset (0,1), \; \left( L^{\partial \Gamma_{\alpha_k}}(X) \right)_{k \in \mathbb{Z}}$ are specified in Remark \ref{r14}. Moreover, it holds for any $k \in \mathbb{Z}$ 
\begin{equation*}
  \int_0^t \mathbb{I}_{\partial \Gamma_{\alpha_k}} (X_s) \mathrm{d} L^{\partial \Gamma_{\alpha_k}}_s (X) = L^{\partial \Gamma_{\alpha_k}}_t (X), \quad t \geq0 \; P_z\textup{-a.s.}
\end{equation*}
In particular, if $\mathbb{I}_{\Gamma_{\alpha_k}} \in BV(H,H), \; k\in \mathbb{Z}$, then above we can replace $H_1$ by $H$ and ${ }_{H_1} \langle \cdot, \cdot \rangle_{H^{*}_1}$ by $\langle \cdot, \cdot \rangle$.
\end{thm}
\begin{examples}\label{exam1}
\begin{itemize}
 \item[(i)] Let $\Gamma \subset H$ be closed and convex such that $\mathbb{I}_\Gamma \in BV(H,H_1)$ (for concrete examples of such $\Gamma$ see (ii) or (${\it0}4$)). For $k \in \mathbb{Z}$ set 
 \begin{align*}
  \Gamma_{\alpha_k} :=\begin{cases}\emptyset  &\text{ if } \;k \leq -1,\\  \Gamma  &\text{ if } \;k =0,\\H & \text{ if }\;k \geq1, \end{cases}
 \end{align*}
and 
\begin{align*}
 \gamma_k :=\begin{cases}\frac{1-p}{p}  &\text{ if } \;k \leq 0,\\  1  &\text{ if } \;k \geq1. \end{cases}
\end{align*}
where $p \in (0,1)$ is fixed. \\
Then $\rho= \frac{1-p}{p} \mathbb{I}_\Gamma + \mathbb{I}_{H\setminus \Gamma}$, \eqref{S1}-\eqref{S3} are satisfied, and $(p_k)_{k \in \mathbb{Z}}$ as in \eqref{e4.5} satisfies 
\begin{align*}
 p_k = \begin{cases}\frac{1}{2}  &\text{ if } \;k \leq -1,\\  p  &\text{ if } \;k =0,\\ \frac{1}{2} & \text{ if }\; k \geq1. \end{cases}
\end{align*}
Therefore in this case \eqref{e4.6} takes the form 
\begin{align}\label{e4.7}
 \langle l, X_t-X_0 \rangle &= \int_0^t \langle l, \mathrm{d} W_s^z \rangle - \int_0^t \langle Al, X_s \rangle \mathrm{d} s \\
 &  + ( 2 p -1) \int_0^t { }_{H_1}\langle l, \eta_\Gamma(X_s) \rangle_{H^{*}_1}\, \mathrm{d} L^{\partial \Gamma}_s (X)\,, \quad t \geq 0,\; P_z\text{-a.s.} \notag
\end{align}
\eqref{e4.7} can be seen as the infinite dimensional analogue of the $p$-skew reflected OU-process, i.e. the process which is associated to the closure of 
\begin{equation*}
 \frac{1}{2} \int\limits_{\mathbb{R}} f'(x) g'(x) \left( \frac{1-p}{p} \mathbb{I}_{(-\infty, b)}+ \mathbb{I}_{(b, \infty)} \right) e^{-\beta x^2} \mathrm{d} x\,, \;\; f,g \in C^\infty_0(\mathbb{R}),
\end{equation*}
on $L^2(\mathbb{R}, e^{-\beta x^2} \mathrm{d}x)$, where $p, \beta, b \in \mathbb{R}, p \in (0,1), \beta>0$, are arbitrary, but fixed parameters (see for instance \cite{Tr05}).
\item[(ii)] Let $H= L^2(0,1),\; \Gamma_\alpha:=\{ f \in H | f \geq - \alpha \}, \alpha \in \mathbb{R}, \alpha >0$. Then obviously $\Gamma_\alpha$ is
closed and convex. Let further $A=-\frac{1}{2} \frac{\mathrm{d}^2}{\mathrm{d} r^2}$ with Dirichlet boundary conditions. In this case $e_j = \sqrt{2} \sin (j \pi r), j \in \mathbb{N}$, is the corresponding eigenbasis. 
Defining
\begin{equation*}
 c_j:= (j \cdot \pi )^{\frac{1}{2}+\varepsilon}\,, \; j \in \mathbb{N},
\end{equation*}
where $\varepsilon \in (0, \frac{3}{2}]$, we obtain that $(H_1,H, H^{*}_1)$ as defined in Section \ref{sec2} is a Gelfand triple (cf. \cite[Section 6]{RZZ12}).
Moreover, by \cite[Theorem 6.2, Remark 6.3]{RZZ12} \, $\mathbb{I}_{\Gamma_\alpha} \in BV(H,H_1)\setminus  BV(H,H)$ for any $\alpha >0$. Now choose sequences 
$(\alpha_k)_{k\in \mathbb{Z}},\, (\gamma_k)_{k\in \mathbb{Z}} \subset (0, \infty)$ 
such that the conditions \eqref{S1}-\eqref{S3} are satisfied and apply Theorem \ref{t4.4} to obtain a concrete example of a countably skewed OU-process.
\end{itemize}

\end{examples}
\begin{remark}\label{bouzam}
By \cite[Theorems 6.4 and 6.5]{RZZ12} it is also possible to treat the case $\alpha=0$ in Example \ref{exam1}(ii). For instance, letting
$(\alpha_k)_{k\in \mathbb{Z}},\, (\gamma_k)_{k\in \mathbb{Z}} \subset (0, \infty)$ and $p$ be as in Example \ref{exam1}(i) and choosing $\Gamma=\Gamma_{\alpha_0}=\Gamma_{0}:=\{ f \in H | f \geq 0 \}$. 
This may be treated in forthcoming work.
An attempt to describe the stochastic dynamics of a $p$-skew OU-process under different aspects, for instance using finite-dimensional approximations and Mosco convergence, 
is undertaken in \cite{BouZ14}.
\end{remark}

Michael R\"ockner\\ 
Fakult\"at f\"ur Mathematik\\
Universit\"at Bielefeld\\
Universit\"atsstrasse 25\\
33615 Bielefeld \\
E-mail: roeckner@math.uni-bielefeld.de\\ \\
Gerald Trutnau\\
Department of Mathematical Sciences and \\
Research Institute of Mathematics of Seoul National University,\\
San56-1 Shinrim-dong Kwanak-gu, \\
Seoul 151-747, South Korea,  \\
E-mail: trutnau@snu.ac.kr

\addcontentsline{toc}{chapter}{References}
\begin{thebibliography}{99}

\bibitem{AMMP10} L. Ambrosio, M. Miranda, S. Maniglia, D. Pallara,  {\it BV functions in abstract Wiener spaces}, J. Funct. Anal. 258 (2010), no. 3, 785-813


\bibitem{BouZ14} S.K. Bounebache, L. Zambotti, {\it 
A skew stochastic heat equation}, J. Theoret. Probab. 27 (2014), no. 1, 168-201. 

\bibitem{BDT09}
V. Barbu, G. Da Prato, L. Tubaro, {\it Kolmogorov equation associated to the stochastic reflection problem on a smooth convex set of a Hilbert space}, Ann. Probab. {\bf 37} (2009), 1427-1458.

\bibitem{BDT10}
V. Barbu, G. Da Prato, L. Tubaro, {\it Kolmogorov equation associated to the stochastic reflection problem an a smooth convex set of a Hilbert space II},
Annales de l'Institut Henri Poincar\'{e}, Probabilit\'{e}s et Statistiques, {\bf 47} (2011), no. 3, 699-724. 

\bibitem{BaHsu91} Richard F. Bass, Pei Hsu, {\it Some potential theory for reflecting Brownian motion in H\"older and Lipschitz domains}, Ann. Probab. 19 (1991), no. 2, 486-508.

\bibitem{C98} E. C\'epa, {\it Probl\`eme de Skorohod multivoque}, Ann. Probab. 26 (1998), no. 2, 500-532.



\bibitem{DI93} P. Dupuis, H. Ishii, {\it SDEs with oblique reflection on nonsmooth domains}, Annals of Probability, Vol. 21, No. 1,  (1993), pp 554-580.

\bibitem{DuRa99} P. Dupuis, K. Ramanan, {\it Convex duality and the Skorokhod problem. I, II}, Probab. Theory Related Fields 115 (1999), no. 2, 153–195, 197–236. 

\bibitem{F00}M. Fukushima, {\it BV functions and distorted Ornstein Uhlenbeck processes over the abstract Wiener space},
J. Funct. Anal. 174 (2000), no. 1, 227–249. 

\bibitem{FH01}M. Fukushima, M. Hino, {\it On the space of BV functions and a related stochastic calculus in infinite dimensions}, J. Funct. Anal. {\bf 183} (2001), 245-268.

\bibitem{FOT11} M. Fukushima, Y. Oshima, M. Takeda, {\it Dirichlet forms and symmetric Markov processes}, Second revised and extended edition,
de Gruyter Studies in Mathematics, 19. Walter de Gruyter \& Co., Berlin, 2011. x+489 pp.

\bibitem{HS81} J. M. Harrison,  L. A. Shepp, {\it On skew Brownian motion}, Ann. Probab. 9 (1981), no. 2, 309-313. 

\bibitem{LS84} P.-L. Lions, A.-S. Sznitman, {\it Stochastic differential equations with reflecting boundary conditions}, 37 (1984) no.4, pp 511-533.


\bibitem{LMS81} P.-L. Lions, J.-L. Menaldi, A.-S. Sznitman, {\it Construction de processus de diffusion r\'efl\'echis par p\'enalisation du domaine}, C. R. Acad. Sci. Paris Ser. I Math. 292 (1981), no. 11, 559-562.
 
\bibitem{K87} Jai Heui Kim, {\it Stochastic calculus related to nonsymmetric Dirichlet forms}, Osaka J. Math. 24 (1987), no. 2, 331-371.

\bibitem{LR15}
W. Liu, M. R\"ockner, {\it Stochastic Partial Differential Equations: An Introduction}, Universitext, Springer, to appear, 2015.

\bibitem{MR92}
Z. M. Ma, M. R\"ockner, {\it Introduction to the Theory of (nonsymmetric) Dirichlet Forms}, Universitext, Springer, Berlin 1992.

\bibitem{Mal97} P. Malliavin, {\it Stochastic Analysis}, Grundlehren der Mathematischen Wissenschaften {\bf 313}, Springer, Berlin (1997).

\bibitem{NP92} D. Nualart, E. Pardoux, {\it White noise driven quasilinear SPDEs with reflection}, Probab. Theory Related Fields 93 (1992), no. 1, 77-89.

\bibitem{O13} Y. Oshima, {\it Semi-Dirichlet Forms and Markov Processes}, Walter de Gruyter 2013.

\bibitem{ORT15} Y. Ouknine, F. Russo, G. Trutnau, {\it On countably skewed Brownian motion with accumulation point}, Electron. J. Probab. {\bf 20} (2015), no. 82, 1-27.

\bibitem{PW94} E. Pardoux, R.J. Williams, {\it Symmetric reflected diffusions}, Ann. Inst. H. Poincare Probab. Statist. 30 (1994), no. 1, 13-62.

\bibitem{Pe07} G. Peskir, {\it A change-of-variable formula with local time on surfaces}, Seminaire de Probabilites XL, 69-96, Lecture Notes in Math., 1899, Springer, Berlin, 2007.

\bibitem{Ra11}
J. M. Ramirez, {\it Multi-skewed Brownian motion and diffusion in layered media}, Proc. Amer. Math. Soc. {\bf139} (2011), no. 10, 3739–3752.

\bibitem{R70} D. Revuz, {\it Mesures associ\'ees aux fonctionnelles additives de Markov. I}, Trans. Amer. Math. Soc. 148 1970 501-531.

\bibitem{RY99}
D. Revuz, M. Yor, {\it Continuous Martingales and Brownian Motion}, Springer-Verlag Berlin Heidelberg, 1999.

\bibitem{RS92}
M. R\"ockner, B. Schmuland, {\it Tightness of general $C_{1,p}$ capacities on Banach Space}, J. Funct. Anal. {\bf 108} (1992), 1-12.

\bibitem{RZZ12}
M. R\"ockner, R. Zhu, X. Zhu, {\it The stochastic reflection problem on an infinite dimensional convex set an BV functions in a Gelfand triple},  AOP {\bf 40} (2012), no. 4, 1759-1794.

\bibitem{St99} W. Stannat, {\it The theory of generalized Dirichlet forms and its applications in analysis and stochastics}, Mem. Amer. Math. Soc. 142 (1999), no. 678, viii+101 pp.

\bibitem{Tr00} G. Trutnau,  {\it Stochastic calculus of generalized Dirichlet forms and applications to stochastic differential equations in infinite dimensions}, Osaka J. Math. 37 (2000), no. 2, 315-343.

\bibitem{Tr03}
G. Trutnau, {\it Skorokhod decomposition of reflected diffusions on bounded Lipschitz domains with singular non reflection part}, Probab. Theory Relat. Fields {\bf 127} (2003), no.4, pp. 455-495.

\bibitem{Tr05}
G. Trutnau, {\it Multidimensional skew reflected diffusions, Stochastic Analysis: Classical and Quantum. Perspectives of white noise theory}, World Sci. Publ., Hackensack, NJ, (2005), pp. 228-244.

\bibitem{Wei84}
S. Weinryb, {\it Homog\'{e}n\'{e}isation pour des processus associ\'{e}s \`{a} des fronti\`{e}res perm\'{e}ables}, Ann. Inst. H. Poincaré Probab. Statist. {\bf20} (1984), no. 4, 373-407. 

\bibitem{Z01} L. Zambotti, {\it A reflected stochastic heat equation as symmetric dynamics with respect to the 3-d Bessel bridge}, J. Funct. Anal., 180(1):195-209, 2001.

\bibitem{Z02} L. Zambotti, {\it Integration by parts formulae on convex sets of paths and applications to SPDEs with reflection}, Probab. Theory Related Fields 123 (2002), no. 4, 579-600. 

\end{thebibliography}
\end{document}